\documentclass[12pt]{amsart}
\usepackage{amsmath,amsfonts,amstext,amsthm,amssymb,verbatim,bbm,hyperref}

\usepackage{color}

\usepackage[margin=1.1in]{geometry}

\newtheorem{lemma}{Lemma}

\newtheorem{theorem}[lemma]{Theorem}
\newtheorem{proposition}[lemma]{Proposition}

\newtheorem{corollary}[lemma]{Corollary}
\newtheorem*{thm*}{Theorem}

\newtheorem{definition}[lemma]{Definition}

\theoremstyle{remark}
\newtheorem{remark}[lemma]{Remark}
\newtheorem*{remark*}{Remark}

\numberwithin{equation}{section}
\numberwithin{lemma}{section}

\numberwithin{equation}{section}

\def\dec{\mathrm{dec}}
\def\tr{\operatorname{tr }}

\def\Airy{\operatorname{Ai}}
\def\TW{\operatorname{TW}}

\newcommand{\Ai}{\mathfrak{Ai}}
\newcommand{\E}{\mathbb E}

\newcommand{\ii}{\mathbf{i}}

\renewcommand{\aa}{\mathfrak a}
\newcommand{\bb}{\mathfrak b}

\title{The KPZ equation and moments of random matrices}

\author{Vadim Gorin}

\address{Department of Mathematics,
Massachusetts Institute of Technology,  Cambridge, MA, USA and Institute for
Information Transmission Problems of Russian Academy of Sciences, Moscow, Russia.}
\email{vadicgor@gmail.com}

\author{Sasha Sodin}

\address{School of Mathematical Sciences, Queen Mary University of London, London E1 4NS,
United Kingdom \, and School of Mathematical Sciences, Tel Aviv University, Tel Aviv,
69978, Israel.}

\email{a.sodin@qmul.ac.uk}

\begin{document}

\begin{abstract}
 The logarithm of the diagonal matrix element of a high power of a random matrix
 converges to the Cole--Hopf solution of the Kardar--Parisi--Zhang equation in the sense of one--point distributions.
\end{abstract}

\maketitle

\section{Introduction and results}

\subsection{KPZ equation} Let $Z(t, x)$, be the solution of the stochastic heat equation
\begin{equation}\label{eq:she}\frac{\partial Z}{\partial t} = \frac{1}{2} \frac{\partial^2 Z}{\partial x^2} - \dot WZ~,
\quad x\in\mathbb R, \, t\ge 0;\quad
Z(0, x) = \delta(x)~,\end{equation}
where $\dot W$ is the space-time white noise. The
 logarithm $H = - \log Z$ is known as the Cole--Hopf solution of the Kardar--Parisi--Zhang \cite{KPZ}
  equation with the narrow-wedge
 initial condition. Informally, it solves a singular stochastic PDE
\begin{equation}\label{eq:kpz}
 \frac{\partial H}{\partial t}=\frac{1}{2} \frac{\partial^2 H}{\partial x^2} -\frac{1}{2}\left(\frac{\partial H}{\partial x}\right)^2 +\dot W~;
\end{equation}
the properties of the solutions of \eqref{eq:kpz} are discussed in
\cite{Corwin_KPZ,QS,GJ_KPZ}, and a rigorous regularisation  -- in \cite{Hairer_KPZ,
GP_KPZ1,GP_KPZ2}.

\bigskip

We also need a half--line version of \eqref{eq:she}, which was recently studied in
\cite{BBCW, Par}. Introduce $\tilde Z(t, x)$ as the solution to
$$
 \frac{\partial \tilde Z}{\partial t} = \frac{1}{2} \frac{\partial^2 \tilde Z}{\partial x^2} - \dot W\tilde Z~, \quad x\ge 0,\, t\ge 0;\quad
 \left(\frac{\partial \tilde Z}{\partial x}+\frac{1}{2} \tilde Z\right)_{x=0}=0,\quad t>0;\quad
\tilde Z(0, x) = \delta(x).
$$
$\tilde H = -\log\tilde Z$ is the Cole--Hopf solution for KPZ with a Neumann
boundary condition at $x=0$.

\subsection{Random matrix edge}
 On the other side of the picture, let $X_N$ be a $N\times N$ matrix of independent,
 identically distributed  Gaussian random variables. In $\beta=1$ case they are real $N(0,2)$
 and in $\beta=2$ case they are complex with real and imaginary part independent $N(0, 1)$. The distribution of the Hermitian
 matrix
 $M_N=\frac{1}{2}(X_N+ X_N^*)$ is known as GOE at $\beta=1$ and GUE at $\beta=2$. Let
\[ \lambda_{1,N} > \lambda_{2,N} > \cdots > \lambda_{N,N} \]
 be the eigenvalues of $M_N$. Then (see \cite{For1,TW1,TW2} and the textbooks \cite{For},\cite{PSbook})
\begin{equation}\label{eq:toairy}
\left\{ N^{\frac16} \left(\lambda_{j,N} -2\sqrt{N}\right) \right\}_{j=1}^{\infty} \underset{N \to \infty}{\longrightarrow}  \Ai_{\beta}, \quad \beta=1,2,
\end{equation}
where $\Ai_{\beta}$ is the Airy$_{\beta}$ point process, a realisation of which is an
ordered sequence  $\lambda_1 > \lambda_2 > \cdots$ The relation \eqref{eq:toairy} means that the right-most
$k$ rescaled eigenvalues converge in distribution to the $k$ rightmost points of the Airy process,
for any $k$.

For complex matrices, $\Ai_2$ is a determinantal point process on the real line defined by the
kernel
\[ K_{\operatorname{Airy}}(\lambda, \lambda') = \frac{\Airy(\lambda)\Airy'(\lambda')-\Airy'(\lambda)\Airy(\lambda')}{\lambda-\lambda'} \]
For real matrices, $\Ai_1$ is a Pfaffian point process whose kernel is also written
in terms of the Airy function $\Airy(x)$. For both $\beta=1,2$, the distribution of
the first particle $\lambda_1$ is known as the Tracy--Widom$_\beta$ law,
$\TW_\beta$. See further \cite{TW2} for the case $\beta=4$ and \cite{RRV} for
arbitrary $\beta
>0$.

\subsection{Main results} The key notion of our article is the decorated Airy process $\Ai^\dec$.

\begin{definition} \label{Def_decorated_Airy} Let $u_i^j$, $v_i^j$, $i,j=1,2,\dots$ be i.i.d.\ standard real Gaussian random variables.
For $n=1,2,\dots,$ let $W_n$ be a $\mathbb N\times \mathbb N$ Hermitian rank $1$ matrix
$$[W_n]_{a,b}=\begin{cases} u^n_a u^n_b, & \beta=1\\
 \frac12 (u^n_a+\ii v^n_a)(u^n_b-\ii v^n_b),& \beta=2. \end{cases}
$$
$\Ai^\dec_\beta$ is (the distribution of) the random matrix--valued measure
$$
 \sum_{n=1}^{\infty} W_n \delta_{\lambda_n},\quad \{\lambda_n\} \sim \Ai_\beta.
$$
\end{definition}
For each continuous function $f:\mathbb R\to\mathbb C$ such that $ \int_{-\infty}^0 |f(\lambda)|
\sqrt{|\lambda|} \, d\lambda < \infty~$, the integral $\int f(\lambda) \Ai^\dec_\beta (d\lambda)$
is a Hermitian random matrix of size $\mathbb N\times \mathbb N$ with matrix elements
\begin{equation} \label{eq_decorated_Airy}
 \left[\int_{\mathbb R} f(\lambda) \Ai^\dec_\beta (d\lambda)\right]_{a,b}=\sum_{n=1}^\infty [W_n]_{a,b}\, f(\lambda_n), \qquad a,b=1,2,\dots.
\end{equation}

\begin{remark}
 The formula \eqref{eq_decorated_Airy} defines a law of an infinite Hermitian matrix invariant
 under conjugations by matrices from the infinite--dimensional orthogonal/unitary group ($O(\infty)=\bigcup_N O(N)$,
 $U(\infty)=\bigcup_N U(N)$, where we embed the group of dimension $N$ into the group of dimension
 $N+1$ by fixing the last basis vector). In the notation of \cite{OV}, the ergodic decomposition
 of this matrix is given by the point process $f(\Ai_\beta)$ on the $\alpha$--parameters of
 the decomposition. It would be interesting to compute the joint law of several elements of the matrix \eqref{eq_decorated_Airy} explicitly
  (at least for specific choices of $f$).

\end{remark}

\smallskip

The first result asserts that the one--point distributions for the stochastic heat equation at $x=0$ are given by the integrals of exponents against a diagonal element of $\Ai^\dec$.

\begin{theorem}\label{Theorem_KPZ} For each $\alpha > 0$ (but not jointly for several $\alpha$'s),
in distribution
\begin{align}
\label{eq_KPZ_1} &\tilde Z(2 \alpha^3, 0) \exp(\alpha^3/12)
\stackrel{d}{=} 2 \left[\int e^{\alpha \lambda} \Ai^\dec_1(d\lambda)\right]_{1,1}
\\ \notag &\quad =2 \sum_{n=1}^\infty (u^n_1)^2 \exp(\alpha\lambda_n),
 \quad \{\lambda_n\}\sim \Ai_1,
\\  \label{eq_KPZ_2} &Z(2 \alpha^3, 0) \exp(\alpha^3/12)
\stackrel{d}{=} \left[\int e^{\alpha \lambda} \Ai^\dec_2(d\lambda)\right]_{1,1}
\\  \notag &\quad = \sum_{n=1}^\infty \frac{(u^n_1)^2+(v^n_1)^2}{2} \exp(\alpha\lambda_n), \quad \{\lambda_n\}\sim \Ai_2.
\end{align}
\end{theorem}
 The $\beta=2$ part of Theorem
\ref{Theorem_KPZ} is a close relative of the determinantal formula for the Laplace transform of the
one--point distribution of SHE computed in \cite{ACQ,CLDR,Dots,SS}   and
linked to the Airy point process in \cite{BG}. The $\beta=1$ part is related to Laplace transforms
in \cite{BBCW,Par}. The proof of Theorem \ref{Theorem_KPZ} (based on all these results) is
given in Section~\ref{s:proofs}.

One application of Theorem \ref{Theorem_KPZ} is that it makes the $T\to+\infty$
limit for the solution to the KPZ equation immediate, leading to $\TW_\beta$
distribution after proper centering and rescaling.

\medskip

On the random matrix side, the decorated Airy process governs the edge asymptotic behavior of large
random matrices.

\begin{theorem} \label{Theorem_edge} Let $M_N$ be a Hermitian (Wigner) matrix with
independent (up to symmetry) real/complex elements, so that the moments of the
matrix elements $\mathbb E [M_N]_{i,j}^k \overline{[M_N]_{i,j}^\ell}$ with $k + \ell
\leq 4$ match those of  GOE/GUE, respectively, and suppose
that\begin{equation}\label{eq:techn}\begin{split} \sup_{N} \max_{1 \leq i,j \leq N} \mathbb E |[M_N]_{i,j}|^{C_0} <\infty~,
\end{split}\end{equation}
where $C_0$ is the absolute constant from \cite{TV}.
 Then for all $\alpha > 0$
\begin{equation}\label{eq_matrix_element_convergence}
\lim_{N\to\infty} \frac{N}{2}  \left[\left(\frac{M_N}{2\sqrt N}\right)^{2 \lfloor
\alpha N^{2/3} \rfloor}+ \left(\frac{M_N}{2\sqrt N}\right)^{2 \lfloor \alpha N^{2/3}
\rfloor+ 1}\right]_{a,b=1}^\infty  = \int_\mathbb R \exp(\alpha \lambda)
\Ai^\dec_\beta(d\lambda),
\end{equation}
in the sense of the convergence of joint distributions for finitely many $a$, $b$, and $\alpha$'s.
\end{theorem}
 If instead of the matrix elements we deal with the trace of the matrix in the LHS of
 \eqref{eq_matrix_element_convergence}, then Theorem \ref{Theorem_edge} (with different technical conditions)
 turns into the statement of \cite{Sosh} describing the universality of the eigenvalue distribution at the spectral
 edges. The matrix elements, which we now consider,
 contain the information on the eigenvectors in addition to the eigenvalues. If we deal with
 GOE/GUE, then eigenvectors are uniformly distributed on the unit sphere (of $\mathbb R^N$ for $\beta=1$ and of $\mathbb C^N$ for $\beta = 2$, and it
 immediately follows that as $N\to\infty$ their components become independent Gaussians. This is
 the reason for the appearance of $\{u_i^j\}$, $\{v_i^j\}$ in Definition \ref{Def_decorated_Airy}.
 The extension to more general Wigner matrices in Theorem
 \ref{Theorem_edge} uses the four moments theorem of \cite{TV}, which essentially says
 that the asymptotic in the general case is the same as for GOE/GUE. A proof of Theorem
 \ref{Theorem_edge} along these lines is given in Section \ref{s:proofs}.

\medskip
 Combining Theorems \ref{Theorem_KPZ}, \ref{Theorem_edge} we arrive at an intriguing corollary.
\begin{corollary}\label{c:main} Let $Z^\beta= Z$ for $\beta=2$ and $\tilde Z$ for $\beta=1$. The one-point distribution of $Z^\beta$ satisfies
\begin{equation}
\label{eq_KPZ_to_moments}
 \lim_{N\to\infty} \frac{N}{\beta} \left[\left(\frac{M_N}{2\sqrt
N}\right)^{2 \lfloor \alpha N^{2/3} \rfloor}+ \left(\frac{M_N}{2\sqrt N}\right)^{2 \lfloor \alpha
N^{2/3} \rfloor+ 1}\right]_{1,1}\stackrel{d}{=}  Z(2\alpha^3, 0) \exp(\alpha^3/12),
\end{equation}
for any $M_N$ as in Theorem \ref{Theorem_edge} and $\alpha > 0$.\end{corollary}

\begin{remark}\label{rmk:gener}
The conditions on the matrix $M_N$ may be relaxed: The matching of the fourth moment
could probably be dropped using the methods of \cite{KY}. As to the tail decay, the
maximal-generality  condition  for   the eigenvalue convergence \eqref{eq:toairy}
was found in \cite{LeeYin}; we expect it to be sufficient also for the problem
considered here.

In a different direction, passing to a slightly different matrix ensemble (e.g.\ sample covariance matrices or
shifted Wigner matrices), one can obtain a statement similar to Corollary~\ref{c:main} with one,
rather than two, terms in the left-hand side.

In yet another direction, Theorem \ref{Theorem_KPZ} and Corollary \ref{c:main} should generalize to
unitary (or orthogonally) invariant random matrix ensembles, the local eigenvalue statistics of which are discussed in \cite{PS,DG} and the
monographs \cite{D,PSbook}.
\end{remark}

The matrix element $[M_N^k]_{1,1}$  admits the following combinatorial interpretation:
\begin{equation}\label{eq:comb}[M_N^k]_{1,1} = \sum_{j_1, \cdots, j_{k-1}} M_N(1, j_1) M_N(j_1, j_2) \cdots M_N(j_{k-1},1)~.\end{equation}
The tuple $1, j_1, \cdots, j_k, 1$ can be thought of as a path in the complete graph, the product in the right-hand side
of (\ref{eq:comb}) collects the random weights along the path.

On the other hand, the Feynman--Kac formula for the Stochastic Heat Equation identifies $Z(t,x)$
 with a partition function of a Brownian directed polymer, see \cite{AKQ1} for a detailed treatment.
 This continuous polymer is a universal limit for discrete directed polymers in the intermediate disorder regime \cite{AKQ2}.
 Such identification
 bears similarities with \eqref{eq:comb}, however, note that there is no clear spatial structure in the latter.
 It would be interesting to find a more direct  relation between these models.

\bigskip

A $1d$  spatial structure can be introduced to GOE/GUE by passing to tridiagonal models of \cite{DE}.
Consider the matrix
$$
\widetilde M_N=\left(\begin{array}{ccccc}
\aa(N) & \bb(N-1) & 0 & \cdots & 0 \\
\bb(N-1) & \aa(N-1) & \bb(N-2) & \ddots & \vdots \\
0 & \bb(N-2) & \aa(N-2) & \ddots & 0 \\
\vdots & \ddots & \ddots & \ddots & \bb(1) \\
0 & \cdots & 0 & \bb(1) & \aa(1)
\end{array}\right),
$$
where all $\aa(m)$, $m=1,2,\dots$, have the normal
distribution $N(0,2/\beta)$, and the $\bb(m)$, $m=1,2,\dots$, are $\beta^{-1/2}$
multiples of $\chi$-distributed random variables with parameters $\beta m$.
Here, the density of the $\chi$ distribution
with parameter $a$ on $\mathbb R_{\ge 0}$ is
\[
\frac{2^{1-k/2}}{\Gamma(a/2)}\,x^{a-1}\,e^{-x^2/2}, \quad x>0.
\]
For $\beta=1,2$, $\widetilde M_N$ are obtained from GOE/GUE by the
tridiagonalization procedure, which keeps the spectral measure of the $(1,1)$ matrix
element (i.e.\ the functional $f\mapsto \langle f(M_N) e_1,e_1\rangle$, where $e_1$
is the first basis vector)
 unchanged. This implies that the laws of
$[M_N^k]_{1,1}$ and $[\widetilde M_N^k]_{1,1}$ coincide (even jointly for several
$k$), and therefore, $M_N$ can be replaced with $\widetilde M_N$ in Corollary
\ref{c:main}. The large moments of $\widetilde M_N$ were studied in detail in
\cite{GSh}. A slight generalization of the results therein leads to an alternative
form of \eqref{eq_KPZ_to_moments}.

\begin{proposition} The random variables of
formula \eqref{eq_KPZ_to_moments} and Theorem \ref{Theorem_KPZ} coincide with the distributional limit
 \begin{equation}
\label{eq_KPZ_to_tridiag}
 \lim_{N\to\infty} \frac{N}{\beta} \left[\left(\frac{\widetilde M_N}{2\sqrt
N}\right)^{2 \lfloor \alpha N^{2/3} \rfloor}+ \left(\frac{\widetilde M_N}{2\sqrt
N}\right)^{2 \lfloor \alpha N^{2/3} \rfloor+ 1}\right]_{1,1},
\end{equation}
and is also given by
\begin{equation}\label{def:kernel}
\frac{1}{\beta\alpha  \sqrt{\pi \alpha}} \E_{\mathfrak e} \Biggl[\exp\left(
-\frac{1}{2}\int_0^{2\alpha} \mathfrak
e(t)\,\mathrm{d}t+\frac{1}{\sqrt{\beta}}\,\int_0^\infty L_a(\mathfrak
e)\,\mathrm{d}W(a)\right)\Biggr],
\end{equation}
where $\mathfrak e(t)$ is a standard Brownian excursion going from $0$ to $0$ in
time $2\alpha$, $L_a$ is its local time, and $W(a)$ is a standard Brownian motion.
\end{proposition}
\begin{remark}
 The limit \eqref{eq_KPZ_to_tridiag} exists and is given by \eqref{def:kernel} for quite general choice of distributions
 for $\aa(n)$, $\bb(n)$ in the definition of $\widetilde M_N$, see \cite[Assumption 2.1]{GSh}.
\end{remark}

Continuing the discussion started after Remark~\ref{rmk:gener}, we see that the expression \eqref{def:kernel} can be treated
as a partition function of a certain polymer. However, an important difference to the KPZ is that the white noise $\mathrm{d} W(a)$ in
\eqref{def:kernel} is one--dimensional, while $\dot W$ in the Stochastic Heat Equation is 2--dimensional space-time white noise. It would
be interesting to find a direct proof of the distributional identity between \eqref{def:kernel} and the expressions of Theorem \ref{Theorem_KPZ}.

\subsection*{Acknowledgements} V.G.~is partially supported by the NSF grant DMS-1664619,
by the Sloan Research Fellowship, and by NEC Corporation Fund for Research in
Computers and Communications. S.S.~is partially supported by the European Research
Council starting grant 639305 (SPECTRUM) and by a Royal Society Wolfson Research Merit Award.

\section{Proofs}\label{s:proofs}

\begin{proof}[Proof of Theorem \ref{Theorem_KPZ}]
For $\beta=2$ our starting point is the following exact relation between the distribution of $Z(t, 0)$
and the Airy process which was recently  found in \cite[Theorem 2.1]{BG} as a reformulation of the results of
 \cite{ACQ,Dots,CLDR,SS}:
for each $\alpha > 0$ and  $u \geq 0$,
\begin{equation}\label{eq:bg}
 \mathbb E \exp \left[ - u Z(2 \alpha^3, 0) \exp(\alpha^3/12) \right]
 = \mathbb E\left[ \prod_{k=1}^\infty \frac{1}{1 + u \exp(\alpha \lambda_k)}\right], \quad \{\lambda_i\}\sim \Ai_2.
\end{equation}
An analogue for $\beta=1$ is \cite[Theorem B]{BBCW}, \cite[Corollary 1.3]{Par}: for each $\alpha>0$ and $u\geq 0$,
\begin{equation}\label{eq:bbcw}
 \mathbb E \exp \left[ - u \tilde Z(2 \alpha^3, 0) \exp(\alpha^3/12) \right]
 = \mathbb E \left[\prod_{k=1}^\infty \frac{1}{\sqrt{1 + 4 u \exp(\alpha \lambda_k)}}\right], \quad \{\lambda_i\}\sim \Ai_1.
\end{equation}
By a uniqueness theorem for the Laplace transform, the result of Theorem \ref{Theorem_KPZ} would follow from
\begin{equation}\label{eq_Laplace_need}
 \mathbb{E} \exp \left[- u \times \text{LHS of \eqref{eq_KPZ_1} or \eqref{eq_KPZ_2}} \right] =
 \mathbb{E} \exp \left[- u \times \text{RHS of \eqref{eq_KPZ_1} or \eqref{eq_KPZ_2}}  \right]
\end{equation}
for all $u \geq 0$. For the left-hand side, the answer is given by \eqref{eq:bg}, \eqref{eq:bbcw}. For the right--hand side,
 the squared standard Gaussian and sum of the squares of two independent Gaussians are particular cases of $\chi^2_{\beta}$ distribution at $\beta=1,2$.
 The Laplace transform of $\chi_\beta^2$ is given by
$$
 \E \exp\Bigl( -v\chi^2_\beta \Bigr)=\frac{1}{(1+2v)^{\beta/2}},\quad v>-1/2.
$$
Therefore, using the factorisation of the Laplace transform of the sum of independent random variables factorizes, we have for $\beta=1,2$
\begin{equation}
 \E \exp\left(-v \beta \left[\int e^{\alpha \lambda} \Ai^\dec_\beta(d\lambda)\right]_{1,1} \right)=
 \mathbb E \left[\prod_{k=1}^\infty \frac{1}{(1 + 2 v \exp(\alpha \lambda_k))^{\beta/2}}\right]~, \quad v >0~,
\end{equation}
where $\{\lambda_i\}\sim \Ai_\beta$. Setting $v=2u/\beta^2$, we arrive at \eqref{eq:bg} and \eqref{eq:bbcw}.
\end{proof}

\begin{proof}[Proof of Theorem~\ref{Theorem_edge}]
First consider the case of GOE/GUE. We first observe that
\begin{equation}\label{eq:toairy_moments}
\lim_{N\to\infty} \frac{1}{2}\tr  \left[\left(\frac{M_N}{2\sqrt N}\right)^{2 \lfloor
\alpha N^{2/3} \rfloor}+ \left(\frac{M_N}{2\sqrt N}\right)^{2 \lfloor \alpha N^{2/3}
\rfloor+ 1}\right]   = \int_\mathbb R \exp(\alpha \lambda) \Ai_\beta(d\lambda)
\end{equation}
in distribution.
Although \eqref{eq:toairy_moments} is not a formal consequence of \eqref{eq:toairy},
it can be deduced from the latter using a simple estimate on the mean eigenvalue
density of the GOE/GUE. (The relation \eqref{eq:toairy_moments} is also proved for more general
Wigner matrices in \cite{Sosh}.)

Further, the eigenbasis of the GOE/GUE is independent of the spectrum (as follows
from the orthogonal/unitary invariance of the GOE/GUE probability density). The
matrix of coordinates of eigenvectors is a uniformly--random (i.e.\
Haar-distributed) element $O$/$U$ of the $N$--dimensional orthogonal/unitary group. As
$N\to\infty$, the matrix elements of $O$/$U$ multiplied by $\sqrt{N}$ become i.i.d.\
standard real/complex Gaussians. This is a
folklore fact which can be proved, for example, by sampling $O$/$U$ through the
Gram--Schmidt orthogonalization applied to GOE/GUE. We refer to  \cite{DEL} for
more details and quantitative estimates, and to \cite{DF} for a historical
discussion.

At this point, we write the $(a,b)$--th matrix element in the LHS of
\eqref{eq_matrix_element_convergence} through the eigenvalue--eigenvector expansion
as:
\begin{equation}
\label{eq_ev_ev}
 \frac{N}{2} \sum_{j=1}^N z_a^j \overline{z_b^j}\left(  \left(1+ \frac{\lambda_j-2\sqrt{N}}{2\sqrt N} \right)^{2\lfloor \alpha N^{2/3}\rfloor}+
 \left(1+ \frac{\lambda_j-2\sqrt{N}}{2\sqrt N} \right)^{2\lfloor \alpha
 N^{2/3}\rfloor+1}\right),
\end{equation}
where $\lambda_j$, $j=1,\dots,N$, are ordered eigenvalues of $M_N$ and $z^j_i$ are
 (real or complex) coordinates of corresponding eigenvectors. The convergence of
 $N^{1/6}(\lambda_j-2\sqrt{N})$ and $N^{1/6} (\lambda_{N-j}+2\sqrt N)$ to points of
 two (independent) $\Ai_\beta$ point processes together with convergence of $N z_a^j
 \overline{z_b^j}$ to products of independent Gaussian random variables, implies
 that terms of \eqref{eq_ev_ev} converge to \eqref{eq_decorated_Airy} with
 $f(\lambda)=\exp(\alpha\lambda)$. To justify the exchange of  summation with the limit, we take conditional expectation (conditioned on the eigenvalues)
 and use \eqref{eq:toairy_moments}.  This implies \eqref{eq_matrix_element_convergence}.

\smallskip

Now we pass to the case of general $M_N$. Let us consider the $(1,1)$ matrix element
with a fixed $\alpha$ (the proof for the joint distribution of several matrix
elements with different values of $\alpha$ only adds indices to the notation). We
rely on \cite[Theorem 8]{TV}, which we use in the following form. For a Hermitian
matrix $M$, denote by $p_j(M)$ the squared absolute value of the first coordinate of
the eigenvector corresponding to the $j$-the eigenvalue (ordered from the largest to
the smallest one).

\begin{thm*}[Tao-Vu]For any two Wigner matrices $M, M'$ satisfying the assumptions of Theorem~\ref{Theorem_edge} with the same $\beta$, there exists a small $\delta > 0$ such that the following holds. Let $k \leq N^\delta$, and let $G\in C^5(\mathbb R^k \times \mathbb R^k)$
be such that
\begin{equation}\label{eq:tvcond} \sup_{x \in \mathbb R^k \times\mathbb C^{k}} \max_{0 \leq i \leq 5} \|\nabla^i G(x)\| \leq N^{\delta}~. \end{equation}
Then for all sufficiently large $N$, any $1 \leq j_1, \cdots, j_k \leq N$
\[ \left| \E G((\sqrt{N}\lambda_{j_\alpha}(M))_{\alpha=1}^k, (N p_{j_\alpha}(M))_{\alpha=1}^k) - \E G((\sqrt{N}\lambda_{j_\alpha}(M'))_{\alpha=1}^k, (N p_{j_\alpha}(M'))_{\alpha=1}^k) \right| \leq N^{-\delta}~.\]
\end{thm*}

The application of the theorem closely follows the strategy of the proof of
\cite[Theorem~7]{TV}. Let $M'$ be sampled from the GOE or GUE for $\beta=1$ or $2$.
Set $k = \lfloor N^{\frac{\delta}{100}} \rfloor$. Pick a smooth bump function $\chi
\in C^\infty(\mathbb C)$ which is identically one for $[-\frac12, \frac12]$ and
vanishes outside $[-1,1]$. Then let
\begin{align*}
G_1(\vec{x}, \vec{y}) &= \prod_{j=1}^k \chi\left(\frac{y_j}{k}\right), \qquad
G_2(\vec{x}, \vec{y}) = \chi\left(\left(\frac{x_1}{2N} - 1\right)^2 N^{\frac43}
(\log\log N)^{-2} \right),\\
G_3(\vec{x}, \vec{y}) &= 1 - \chi\left(\left(\frac{x_k}{2N} -
1\right)^2 N^{\frac43} k^{-\frac12}\right).
\end{align*}
These three functions satisfy (\ref{eq:tvcond}). For the GOE/GUE,
\[ \mathbb E G_i\left((\sqrt{N}\lambda_j(M'))_{j=1}^k, (N p_j(M'))_{j=1}^k\right)  = 1 - o(1)~, \quad i
=1,2,3,\] hence according to the theorem also
\[ \mathbb E G_i\left((\sqrt{N}\lambda_j(M))_{j=1}^k, (N p_j(M))_{j=1}^k\right)  = 1 - o(1)~, \quad i = 1,2,3,\]
and similarly
\[ \mathbb E G_i\left((-\sqrt{N}\lambda_{N+1-j}(M))_{j=1}^k, (N p_{N+1-j}(M))_{j=1}^k\right)  = 1 - o(1)~, \quad i = 1,2,3.\]
In particular, the expressions inside the expectation are equal to $1$ on an event
of asymptotically full probability. For $i=1$ it means that $Np_j(M)$, $Np_{N+1-j}(M)$
are not large for $j\leq k$. For $i=2$ we conclude that $x_1$ is close to $2N$ and $x_N$
is close to $-2N$. For $i=3$ we get that $x_k$ is far enough from $2N$ and $x_{N-k}$
is far enough from $-2N$.

In particular, from the relations corresponding to $i=3$, we conclude that the
contribution of eigenvalues $\lambda_j$, $k<j\leq N-k$, in the expansion
\eqref{eq_ev_ev} is negligible.

Thus, it remains to prove that
\begin{equation} \label{eq:needtoshow}\begin{split}
& \sum_{j = 1}^k \left[\left(\frac{\lambda_j(M)}{2\sqrt
N}\right)^{2 \lfloor \alpha N^{2/3} \rfloor}+ \left(\frac{\lambda_j(M)}{2\sqrt
N}\right)^{2 \lfloor \alpha N^{2/3}+1 \rfloor}\right] \frac{N p_j(M)}{2}   \to  \int_\mathbb R
\exp(\alpha \lambda)
\Ai^\dec_{\beta}(d\lambda),\\
& \sum_{j = N-k+1}^N
\left[\left(\frac{-\lambda_j(M)}{2\sqrt N}\right)^{2 \lfloor \alpha N^{2/3}
\rfloor}- \left(\frac{-\lambda_j(M)}{2\sqrt N}\right)^{2 \lfloor \alpha N^{2/3}+1
\rfloor}\right] \frac{N p_j(M)}{2}    \to 0,
\end{split}\end{equation}
in distribution as $N \to \infty$, and these relations are already established for $M'$. Take $m$
satisfying $N^{\frac23}(\log \log N)^{-1} \leq m \leq N^{\frac23} \log \log N$, let $\phi$  be a smooth
function  with bounded derivatives up to order $5$, and set
\begin{multline*}
G(x_1, \cdots, x_{2k}, y_1, \cdots, y_{2k}) = \phi\left( \sum_{j=1}^{2k}
\left(\frac{x_j}{2N}\right)^{m} y_j \right) \\ \times \prod_{i=1}^3 \left[ G_i(x_1,
\cdots, x_k, y_1, \cdots, y_k) G_i(-x_{k+1}, \cdots, -x_{2k}, y_{k+1}, \cdots,
y_{2k}) \right]
 \end{multline*}
It also satisfies \eqref{eq:tvcond}, whence we can apply the Tao--Vu theorem with
$$
j_\alpha = \begin{cases}
\alpha~, &1\leq \alpha \leq k, \\
N+1-\alpha~, &k< \alpha \leq 2k,
\end{cases}
$$
to reduce $\E G$ for $M$ to that of $M'$. Since $\phi$ is arbitrary, we conclude
that \eqref{eq:needtoshow} holds.
\end{proof}


\begin{thebibliography}{BBCW}

\bibitem[ACQ]{ACQ} Amir, G.; Corwin, I.; Quastel, J.
Probability distribution of the free energy of the continuum directed random polymer in 1+1 dimensions.
Comm.\ Pure Appl.\ Math.\ 64 (2011), no.~4, 466--537.

\bibitem[AKQ1]{AKQ1} Alberts, T.; Khanin, K.; Quastel, J.
The continuum directed random polymer.
J. Stat. Phys. 154 (2014), no. 1--2, 305--326.

\bibitem[AKQ2]{AKQ2} Alberts, T.; Khanin, K.; Quastel, J.
The intermediate disorder regime for directed polymers in dimension
1 + 1. Ann. Probab. 42, 1212--1256 (2014)

\bibitem[BBCW]{BBCW}  Barraquand, G.;   Borodin, A.; Corwin, I.;  Wheeler, M.
Stochastic six-vertex model in a half-quadrant and half-line open ASEP.
arXiv:1704.04309


\bibitem[BG]{BG} Borodin, A.; Gorin, V.
Moments match between the KPZ equation and the Airy point process.
SIGMA Symmetry Integrability Geom. Methods Appl. 12 (2016), Paper No. 102, 7 pp.

\bibitem[BY]{BY} Bourgade, P.; Yau, H.-T..
The eigenvector moment flow and local quantum unique ergodicity.
Comm.\ Math.\ Phys. 350 (2017), no.~1, 231--278.

\bibitem[CDR]{CLDR} Calabrese, P.; Le Doussal, P.; Rosso, A.
Free-energy distribution of the directed polymer at high temperature.
Euro.\ Phys.\ Lett.~90 (2010), no.~2, 20002.

\bibitem[C]{Corwin_KPZ} Corwin, I. The Kardar-Parisi-Zhang equation and universality class, arXiv:1106.1596

\bibitem[DF]{DF} Diaconis, P.; Freedman, D.
A dozen de Finetti-style results in search of a theory.
Ann.\ Inst.\ H.~Poincar\'e Probab.\ Statist. 23 (1987), no.\ 2, suppl., 397--423.

\bibitem[DEL]{DEL} Diaconis, P.;  Morris, E.; and Lauritzen, S.
 Finite De Finetti Theorems in Linear Models and Multivariate Analysis.
 Scandinavian Journal of Statistics 19, no. 4 (1992): 289-315.

\bibitem[D]{D} Deift, P.\ A.,
Orthogonal polynomials and random matrices: a Riemann-Hilbert approach. Courant
Lecture Notes in Mathematics, 3. New York University, Courant Institute of
Mathematical Sciences, New York; American Mathematical Society, Providence, RI,
1999.

\bibitem[DG]{DG} Deift, P.; Gioev, D.
Universality at the edge of the spectrum for unitary, orthogonal, and symplectic ensembles of random matrices.
Comm. Pure Appl. Math. 60 (2007), no. 6, 867--910.

\bibitem[Do]{Dots} Dotsenko, V.
Bethe ansatz derivation of the Tracy-Widom distribution for one-dimensional directed polymers.
Euro.\ Phys.\ Lett.~90 (2010), no.~2, 20003.

\bibitem[DE]{DE} Dumitriu, I.; Edelman, A. Matrix models for beta ensembles. \textit{J. Math. Phys.} \textbf{43} (2002), 5830--5847.

\bibitem[Fo1]{For1} Forrester, P.\ J. The spectral edge of random matrix ensembles,
Nucl. Phys. B. 402, 709--728 (1994).

\bibitem[Fo2]{For} Forrester, P.\ J. \textit{Log-gases and Random Matrices}. Princeton University Press, 2010.


\bibitem[GS]{GSh} Gorin, V.; Shkolnikov, M.
Stochastic Airy semigroup through tridiagonal matrices.
arXiv:1601.06800 (to appear in Ann.\ Probab.)

\bibitem[GJ]{GJ_KPZ} Gon\c{c}alves P.; Jara M. Nonlinear fluctuations of weakly asymmetric interacting particle systems.
Arch. Rational Mech. Anal., 212(2):597--644, 2014


\bibitem[GP1]{GP_KPZ1} Gubinelli, M.; Perkowski, N.
KPZ Reloaded,
 Communications in Mathematical Physics
January 2017, Volume 349, Issue 1, pp 165--269, arXiv:1508.03877


\bibitem[GP2]{GP_KPZ2} Gubinelli, M.; Perkowski, N. Energy solutions of KPZ are unique. ArXiv preprint arXiv:1508.07764,
2015

\bibitem[H]{Hairer_KPZ}  Hairer, M. Solving the KPZ equation. Annals of Mathematics, 178(2):559--664, 2013


\bibitem[KPZ]{KPZ} Kardar, M.; Parisi, G.; Zhang, Y.~C. Dynamic Scaling of Growing Interfaces, Phys. Rev.
Lett., 56 no. 9, 889--892 (1986).

\bibitem[KY]{KY} Knowles, A.; Yin, J.
Eigenvector distribution of Wigner matrices.
Probab.\ Theory Related Fields 155 (2013), no.\ 3--4, 543--582.

\bibitem[LY]{LeeYin} Lee, J.~O.; Yin, J.
A necessary and sufficient condition for edge universality of Wigner matrices.
Duke Math. J. 163 (2014), no. 1, 117--173.

\bibitem[OV]{OV} Olshanski, G.; Vershik, A. Ergodic unitarily invariant measures on the space of infinite Hermitian matrices,
In: Contemporary Mathematical Physics. F. A. Berezin's memorial volume. Amer. Math. Transl. Ser. 2, vol. 175 (R. L. Dobrushin et al., eds), 1996, pp. 137--175,
arXiv:math/9601215

\bibitem[Pa]{Par} Parekh, S.
The KPZ Limit of ASEP with Boundary.
arXiv:1711.05297

\bibitem[PS1]{PS} Pastur, L.; Shcherbina, M.
On the edge universality of the local eigenvalue statistics of matrix models.
Mat. Fiz. Anal. Geom. 10 (2003), no. 3, 335--365.

\bibitem[PS2]{PSbook} Pastur, L.; Shcherbina, M.
 Eigenvalue distribution of large random matrices.
 Mathematical Surveys and Monographs, 171. American Mathematical Society, Providence, RI, 2011. xiv+632 pp. ISBN: 978-0-8218-5285-9

 \bibitem[RRV]{RRV} Ram{\'\i}rez, J.~A.; Rider, B.; Vir\'ag, B.;
Beta ensembles, stochastic Airy spectrum, and a diffusion.
J.\ Amer.\ Math.\ Soc.\ 24 (2011), no.\ 4, 919--944.

\bibitem[QS]{QS} Quastel, J.; Spohn, H. The One-Dimensional KPZ Equation and Its Universality, J Stat Phys (2015) 160:965--984


\bibitem[SS]{SS} Sasamoto, T.; Spohn, H.
One-dimensional Kardar--Parisi--Zhang equation: an exact solution and its universality.
Phys.\ Rev.\ Lett.~104 (2010), 230602.

\bibitem[So]{Sosh} Soshnikov, Alexander.
 Universality at the edge of the spectrum in Wigner random matrices. Comm. Math. Phys. 207 (1999), no. 3, 697--733.

 \bibitem[TV]{TV} Tao, T.; Vu, V.
Random matrices: universal properties of eigenvectors.
Random Matrices Theory Appl.\ 1 (2012), no.~1, 1150001, 27 pp.

\bibitem[TW1]{TW1} Tracy, C.; Widom, H.
Level-spacing distribution and Airy kernel,
Commun. Math. Phys. 159, 151--174 (1994).

\bibitem[TW2]{TW2} Tracy, C.; Widom, H.
On orthogonal and symplectic matrix ensembles, Commun. Math. Phys. 177, 727--754 (1996).

\end{thebibliography}
\end{document}